\documentclass[11pt]{amsart}
\usepackage{amsmath,amsfonts,amssymb,amsthm,amscd,latexsym,euscript,verbatim}
\usepackage[all]{xy}
\usepackage{graphicx}
\usepackage{color}
\makeatletter
\def\section{\@startsection{section}{1}%
  \z@{1.1\linespacing\@plus\linespacing}{.8\linespacing}%
  {\normalfont\Large\scshape\centering}}
\makeatother

\theoremstyle{plain}

\newtheorem*{conj*}{Root Groups Conjecture}
\newtheorem*{thm1.2}{(1.2) Theorem}
\newtheorem*{thm1.3}{(1.3) Theorem}
\newtheorem*{thm1.4}{(1.4) Theorem}
\newtheorem*{prop*}{Proposition}

\newtheorem{prop}{Proposition}[section]

\newtheorem{thm}[prop]{Theorem}
\newtheorem{cor}[prop]{Corollary}
\newtheorem{lemma}[prop]{Lemma}

\theoremstyle{definition}

\newtheorem{definition}[prop]{Definition}
\newtheorem{hypothesis}[prop]{Hypothesis}
\newtheorem*{Def*}{Definition}

\newtheorem{notation}[prop]{Notation}
\newtheorem*{notation*}{Notation}
\newtheorem{remark}[prop]{Remark}

\newcommand{\Cen}{\rm Cen}
\newcommand{\cala}{\mathcal{A}}

\newcommand{\calg}{\mathcal{G}}

\newcommand{\qq}{\mathbb{Q}}

\newcommand{\zz}{\mathbb{Z}}

\newcommand{\ga}{\alpha}
\newcommand{\gb}{\beta}
\newcommand{\gc}{\gamma}

\renewcommand{\gg}{\gamma}

\newcommand{\gd}{\delta}

\newcommand{\gre}{\epsilon}
\newcommand{\gl}{\lambda}

\newcommand{\sminus}{\setminus}
\newcommand{\lan}{\langle}
\newcommand{\ran}{\rangle}

\newcommand{\Tr}{{\rm Tr}}

\numberwithin{equation}{section}

\begin{document}
\title[Sharply $2$-transitive groups of characteristic~$0$]{\emph{Addendum to} Sharply $2$-transitive groups  of characteristic~$0$}
\author[Malte Scherff,  Katrin Tent]{Malte Scherff\qquad Katrin Tent}

\address{Malte Scherff, Katrin Tent \\
         Mathematisches Institut \\
         Universit\"at M\"unster \\
	 Einsteinstrasse 62\\
         48149 M\"unster \\
         Germany}
\email{tent@wwu.de}


\begin{abstract}
In this short note we show how to modify the construction of non-split sharply $2$-transitive groups of characteristic~$0$  given in \cite{RT} to allow for arbitrary fields of characteristic $0$. 
\end{abstract}

\date{\today}
\maketitle
%%%%%%%%%%%%%%%%%%%%%%%%%%%%%%%%%%%%%%%
%%%%%%%%%%%%%%%%%%%%%%%%%%%%%%%%%%%%%%%%
%%%%%%%%%%%%%%%%%%%%%%%%%%%%%%%%%%%
%section1
\section{Introduction}
%%%%%%%%%%%%%%%%%%%%%%%%%%%%%%%%%%%%%%%%%
%%%%%%%%%%%%%%%%%%%%%%%%%%%%%%%%%%%%%
%%%%%%%%%%%%%%%%%%%%%%%%%%%%%%%%%%%%

The first sharply non-split sharply $2$-transitive groups in characteristic $0$ were constructed in \cite{RT}. However, the construction given there only works when starting from the group $AGL(1,\mathbb{Q})= \mathbb{Q}_+\rtimes \mathbb{Q}^*$.
We modify the construction given there in order to prove:

\begin{thm}\label{t:main}
For any field $\mathbb{K}$ of characteristic $0$ the group $AGL(1,\mathbb{K})= \mathbb{K}_+\rtimes \mathbb{K}^*$ can be embedded into a sharply $2$-transitive group of characteristic $0$ not containing any regular normal subgroup.
\end{thm}

To prove Theorem~\ref{t:main} we introduce the following equivalence relation (replacing the equivalence relation given in \cite{RT}): for
any group $G$ and involution $j\in G$ we say that involutions
$s,t\in G$ are \emph{equivalent relative to $j$} (and write $s\approx_j t$) if $\Cen (js)=\Cen(jt)$.

The following proposition replaces Proposition 1.3 of \cite{RT} and provides the induction step for the proof of Theorem~\ref{t:main}

%%%%%%%%%%%%%%%%%%%%%%%%%%%%%%%%%%%%%%%%%%%%%%%%%%%%%%%%
%%%%%%%%%%%%%%%%%%%%%%%%%%%%%%%%%%%%%%%%%%%%%%%%%%%%%%%%%
\begin{prop}\label{p:main}
%%%%%%%%%%%%%%%%%%%%%%%%%%%%%%%%%%%%%%%%
%%%%%%%%%%%%%%%%%%%%%%%%%%%%%%%%%%%%%%%

Let $G$ be a group containing involutions $j,t$ and $t'$ with $j^{t'}=t$ and $A={\Cen}_G(j)$. Assume that $G, j, t, t'$ and $A$ satisfy assumptions (1) -- (3) of Theorem 1.1. in \cite{RT} and furthermore:
\begin{itemize}

\item[{(4')}] for any involution $s$ with $s\approx_j t$  there is some $a\in A$ such that $s=t^a$.

\item[{(5')}] for any involution $s\neq j$, we have ${\Cen}(js)=\{1\}\cup \{js'\colon  s'\approx_j s\}$.

\item[{(6')}] for any involution $s\notin AtA, s\neq j,$ there is an involution $s'\in G$ with $ s'\approx_j s $ such that ${\Cen}(js)=\langle js'\rangle$.
\end{itemize}

Then for any involution $v\in G$ with $v\neq j$ there exists  an extension $G_1$ of $G$ such that for $A_1={\Cen}_{G_1}(j)$  there exists some $f\in A_1$ with $t^f=v$ and conditions $(1) - (3)$ and $(4') - (6')$ continue to hold with $G_1$ and $A_1$ in place of $G,A$.
\end{prop}

Note that for the group  $AGL(1,\mathbb{Q})$ the new equivalence relation agrees with the one given in \cite{RT}.
It is easy to see exactly as in \cite{RT} that for any field $\mathbb{K}$ of characteristic $0$, the group $AGL(1,\mathbb{K})$ satisfies
properties (1) -- (3) and (4') -- (6').

Using the following lemma, the proof of Proposition 1.3 in \cite{RT} carries over verbatim to this setting.

\begin{lemma}\label{l:equiv}
In the situation of Proposition~\ref{p:main}  we have ${\Cen}(js)={\Cen}((js)^n)$  for any involution $s\in G$.  In particular, for involutions $s, s'\in G$ and $n, m \in\mathbb{Z}$ such that $(js)^n=(js')^m$ we have $s\approx_j s'$.
\end{lemma}
\begin{proof}
Since $js$ centralizes $(js)^n$, assumption (5') implies $s\approx_j j(js)^n$. The second part follows directly from this.
\end{proof}

As in \cite{RT} we also see that for any field $\mathbb{K}$ of characteristic $0$, the group $AGL(1,\mathbb{K})*\mathbb{Z}$ satisfies properties (1) -- (3) and (4') -- (6').
Now the proof of Theorem~\ref{t:main} follows exactly as in \cite{RT}.

%%%%%%%%%%%%%%%%%%%%%%%%%%%%%%%%%%%%%
%%%%%%%%%%%%%%%%%%%%%%%%%%%%%%%%%%%%%%%
%%%%%%%%%%%%%%%%%%%%%%%%%%%%%%%%%%%%%%%%

\end{document}